\documentclass[11pt]{amsart}
\usepackage{amsmath}

\DeclareFontEncoding{OT2}{}{} 

\newcommand{\bA}{{\mathbb A}}
\newcommand{\bN}{{\mathbb N}}
\newcommand{\N}{{\mathbb N}}

\newcommand{\bC}{{\mathbb C}}
\newcommand{\F}{{\mathbb F}}

\newcommand{\bG}{{\mathbb G}}

\newcommand{\lra}{\longrightarrow}
\newcommand{\Q}{{\mathbb Q}}

\newcommand{\bR}{{\mathbb R}}
\newcommand{\Z}{{\mathbb Z}}
\newcommand{\bZ}{{\mathbb Z}}
\newcommand{\Qbar}{{\overline{\Q}}}

\newcommand{\Kbar}{{\overline{K}}}

\DeclareMathOperator{\h}{h}

\DeclareMathOperator{\Gal}{Gal}

\newtheorem{theorem}{Theorem}[section]
\newtheorem{lemma}[theorem]{Lemma}
\newtheorem{corollary}[theorem]{Corollary}

\theoremstyle{definition}
\newtheorem{definition}[theorem]{Definition}

\newtheorem{conjecture}[theorem]{Conjecture}

\newtheorem{Claim}[theorem]{Claim}

\theoremstyle{remark}
\newtheorem{remark}[theorem]{Remark}

\title{A Bogomolov type statement for function fields}
\author{Dragos Ghioca}
\address{Dragos Ghioca, Department of Mathematics, University of British Columbia, Vancouver, B.C. V6T 1Z2, Canada}

\email{dghioca@math.ubc.ca}
\begin{document}

\begin{abstract}
Let $k$ be a an algebraically closed field of arbitrary characteristic, and we let $h:\bA^n(\overline{k(t)})\lra \bR_{\ge 0}$ be the usual Weil height for the $n$-dimensional affine space corresponding to the function field $k(t)$ (extended to its algebraic closure). We prove that for any affine variety $V\subset \bA^n$ defined over $\overline{k(t)}$, there exists a positive real number $\epsilon:=\epsilon(V)$ such that if $P\in V(\overline{k(t)})$ and $h(P)<\epsilon$, then $P\in V(k)$.
\end{abstract}

\thanks{
 The research of the author was partially supported by an NSERC grant.
}

\maketitle

\section{Introduction}

In a paper \cite{Lang-int} from 1965, Lang asks the following question: what are the plane irreducible curves $C$ which contain infinitely many points $(x,y)$ where both $x$ and $y$ are roots of unity? It is easy to see that if $C$ is the zero set of an equation of the form $x^my^n = \zeta$, where $m,n\in\Z$ and $\zeta$ is a root of unity, then indeed $C$ contains infinitely many points with both coordinates roots of unity.  An old theorem of Ihara-Serre-Tate-Lang says that indeed $C$ must have the above form.  Essentially they prove that if $C$ contains infinitely many points where both coordinates are roots of unity, then $C$ must be a (multiplicative) translate of a $1$-dimensional torus by a point with both coordinates roots of unity.  This result can be extended to higher dimensional varieties, and even to subvarieties of abelian varieties (the latter was formerly known as the Manin-Mumford Conjecture, proven by Raynaud \cite{Raynaud}). The following theorem is due to Laurent \cite{Laurent}; we write it in the language of algebraic groups, more precisely for $\bG_m^N$, which is the $N$-th cartesian power of the multiplicative group.

\begin{theorem}(Laurent)
 Let $V\subset \bG_m^N$ defined over $\bC$ be an irreducible affine variety which contains a Zariski dense set of torsion points (i.e., points with coordinates roots of unity). Then $V$ is a multiplicative translate of a torus by a torsion point. 
\end{theorem}

The roots of unity are the points of $\Qbar^*$  which have (naive) Weil height equal to $0$.  The Weil height of $x\in\Qbar$ is defined as follows (the set $\Omega_\Q$ stands for all inequivalent absolute values on $\Q$)
$$\h(x):=\frac{1}{[\Q(x):\Q]}\cdot \sum_{\sigma\in\Gal(\Qbar/\Q)} \sum_{v\in \Omega_{\Q}} \log^+|\sigma(x)|_v ,$$
where $\log^+(z):=\log\max\{1,z\}$ for any real number $z$, and for each absolute value of $\Q$, we fix an extension of it to $\Qbar$.    Also, always a sum involving $\sigma(x)$ over all $\sigma\in\Gal(\Qbar/\Q)$ is simply a sum over all the Galois conjugates of $x$. 
Similarly, for any $x_1,\dots,x_N$ contained in a number field $L$  we define
$$\h((x_1,\dots,x_N)):=\frac{1}{[L:\Q]}\cdot \sum_{\sigma\in\Gal(\Qbar/\Q)} \sum_{v\in \Omega_{\Q}} \log\max\{1,|\sigma(x_1)|_v,\dots, |\sigma(x_N)|_v\} .$$
So, Laurent's result yields that if $V\subset \bG_m^N$ contains a Zariski dense subset of points of height equal to $0$, then $V$ is a torsion translate of an algebraic subgroup of $\bG_m^N$. 

The same conclusion holds if one weakens the hypothesis and only asks that $V$ contains a Zariski dense set of points of \emph{small} height; this was initially known as the Bogomolov Conjecture. So, for each $\epsilon \ge 0$, let 
$$S_{\epsilon}:=\left\{ P\in \left(\Qbar^*\right)^N \text{ : } h(P)\le \epsilon \right\}.$$
\begin{conjecture} (Bogomolov)
\label{Bogomolov conjecture for torus}
Let $V\subset \bG_m^N$ be an irreducible subvariety (defined over $\Qbar$) such that for each $\epsilon > 0$, we have that $V(\Qbar)\cap S_\epsilon$ is Zariski dense in $V$. Then $V$ is a torsion translate of an algebraic subgroup of $\bG_m^N$.
\end{conjecture}

The Bogomolov conjecture in the context of abelian varieties was proven by Ullmo \cite{Ullmo} for curves $V$ embedded in their Jacobians, and in the general case of any subvariety $V$ of an abelian variety $A$ by Zhang \cite{Zhang}. Both Ullmo and Zhang proved the Bogomolov conjecture via an equidistribution statement for points of small height on $A$. A generalization of the Bogomolov statement to semi-abelian varieties was obtained by David and Philippon \cite{daph}. 

The case of Bogomolov conjecture for any power of the multiplicative group (see Conjecture~\ref{Bogomolov conjecture for torus})  was first proved by Zhang in \cite{Z2}. Other proofs of the Bogomolov conjecture for $\mathbb{G}_m^n$ were given by Bilu \cite{Bilu} and  Bombieri and Zannier \cite{Bom}. Similarly to Ullmo and Zhang proofs of the Bogomolov Conjecture for abelian varieties, Bilu proved that the probability measures supported on Galois orbits of \emph{generic} algebraic points of height tending to $0$ converge weakly to the Lebesgue measure $\mu$ on  $\mathcal{C}^N$, where $\mathcal{C}$ is the complex unit circle.  
More precisely, if $\{P_n\}_{n\ge 1}\subset \left(\Qbar^*\right)^N$ is a sequence of points with the property that no proper algebraic subgroup of $\bG_m^N$ contains infinitely many $P_n$'s, then for each continuous function $f$, we have
$$\lim_{n\to\infty} \frac{1}{[\Q(P_n):\Q]}\cdot \sum_{\sigma\in\Gal(\Qbar/\Q)} f\left(P_n^\sigma\right) = \int_{\mathcal{C}^N} f d\mu$$

On the other hand, Bombieri and Zannier proof \cite{Bom} of Conjecture~\ref{Bogomolov conjecture for torus} followed a different path. In \cite[Lemma 1]{Bom}, Bombieri and Zannier show that that for any polynomial $f\in\Z[X_1,\dots,X_N]$, and any point $(x_1,\dots,x_N)\in\Qbar$, and for any prime number $p$ sufficiently large, there exists a positive real number $\epsilon(p)$ such that either $\h((x_1,\dots,x_N))\ge \epsilon(p)$, or $f(x_1^p,\dots,x_N^p)=0$. This result allows the authors of \cite{Bom} to conclude that either the points on the hypersurface $Z(f)$ have a height larger than some absolute positive lower bound, or the hypersurface is invariant under the endomorphism of $\bG_m^N$ given by $(X_1,\dots,X_N)\mapsto (X_1^p,\dots,X_N^p)$. In the latter case, one can see that this means $Z(f)$ is a finite union of torsion translates of subtori of $\bG_m^N$.

The approach of Bombieri and Zannier from \cite{Bom} inspired the author to extend their \cite[Lemma 1]{Bom} in positive characteristic by applying the Frobenius map to affine subvarieties of $\bA^N$ defined over $\overline{\F_p(t)}$. This allowed the author to obtain in \cite{Ghioca-bogo3} a Bogomolov type statement for affine varieties defined over $\overline{\F_p(t)}$.  The picture in positive characteristic for the Bogomolov conjecture is much different due to the varieties defined over finite fields. Indeed, if $V\subset\bG_m^N$ is any subvariety defined over $\overline{\F_p}$, then $V$ contains a Zariski dense set of points of height $0$ (all its points with coordinates in $\overline{\F_p}$). So, it is no longer true that only torsion translates of subtori of $\bG_m^N$ contain a Zariski dense set of points with small height; any \emph{constant} subvariety has this property as well. The group structure of the ambient space $\mathbb{G}_m^n$ disappears from the conclusion of a Bogomolov statement for $\mathbb{G}_m^n$; this motivated our approach from \cite{Ghioca-bogo3} in which the ambient space is simply the affine space, and not an algebraic torus as in \cite{Bom}. 

Motivated by a question of Zinovy Reichstein, we consider in this paper the same problem with respect to the height constructed with respect to a function field $K/k$ of arbitrary characteristic. So, let $k$ be an algebraically closed field of arbitrary characteristic, and let $\Omega_{k(t)}$  be all the inequivalent absolute values on $k(t)$. Each $v\in \Omega_{k(t)}$ corresponds either to the place at infinity $v_\infty$, i.e.   
$$v_\infty\left(\frac{f}{g}\right):=\deg(g)-\deg(f),$$
for nonzero $f,g\in k[t]$, or to a point $\alpha\in k$, i.e.,
$$v_\alpha\left(\frac{f}{g}\right):=\text{ord}_{t-\alpha}\left(\frac{f}{g}\right),$$
where $\text{ord}_{t-\alpha}(f/g)$ is the order of vanishing at $\alpha$ of the rational function $f/g$. 

For each finite extension $K$ of $k(t)$, we let $\Omega_K$ be the set of all (inequivalent) places of $K$ which lie above the places of $k(t)$. We normalize each (exponential) valuation $w\in\Omega_K$ so that the function $w:K\lra \bZ$ is surjective. In other words, for each nonzero $x\in k(t)$, and for each place $w\in\Omega_K$ lying above a place $v\in\Omega_{k(t)}$ we have 
$$w(x)=e(w\mid v)\cdot v(x),$$
where $e(w\mid v)$ is the ramification index for $w\mid v$. 
Then for each $x\in K$ we define its height:
$$\h(x):=\frac{1}{[K:k(t)]}\cdot \sum_{v\in \Omega_{k(t)}}\sum_{\substack{w\in\Omega_K\\w\mid v}} \max\{0, - w(x)\} .$$
We note that the above definition is independent of the choice of field $K$ containing $x$ because the places of a function field are \emph{coherent}, i.e., for each finite extensions $k(t)\subset K\subset L$, for each $v\in\Omega_K$, and for each nonzero $x\in K$ we have
\begin{equation}
\label{coherence of the place v}
v(x)=\frac{1}{[L:K]}\cdot \sum_{\substack{w\in\Omega_L\\ w\mid v}}w(x).
\end{equation}
With the notation from \cite{Serre}, the above condition is that $v$ is \emph{defectless}; this  follows from the arguments of \cite[Chapter 1, Section 4]{Serre} (Hypothesis (F) holds for algebras of finite type over
fields and so, it holds for localizations of such algebras; for each $v\in\Omega_K$ we
apply \cite[Propositions 10 and 11]{Serre} to the local ring of $v$).

Similarly, we define for any $n\in\N$, the height of $(x_1,\dots,x_n)\in \bA^n(K)$ be
$$\h((x_1,\dots,x_n)):=\frac{1}{[K:k(t)]}\cdot \sum_{v\in \Omega_{k(t)}}\sum_{\substack{w\in\Omega_K\\w\mid v}} \max\{0,-w(x_1),\cdots -w(x_n)\} .$$
For each $\epsilon \ge 0$ we let
$$S_{\epsilon}:=\{P\in \bA^n(\overline{k(t)})\text{ : }\h(P) \le \epsilon\}.$$
Then our main result is the following.

\begin{theorem} 
\label{main result}
Let $V\subset \bA^n$ be an affine subvariety defined over $\overline{k(t)}$. Let $W\subseteq V$ be the Zariski closure of $V(k)$. Then there exists $\epsilon>0$ such that for all $P\in (V\setminus W)(\overline{k(t)})$, we have $\h(P)\ge \epsilon$.
\end{theorem}

The following result is an alternative reformulation.
\begin{theorem}
\label{main result 0}
Let $V\subset \bA^n$ be an affine subvariety defined over $\overline{k(t)}$. If for each $\epsilon >0$, the subset $V(\overline{k(t)})\cap S_\epsilon$ is Zariski dense in $V$, then $V$ is defined over $k$.
\end{theorem}

\begin{remark}
\label{R:extension}
The result of Theorem~\ref{main result} (and its reformulation) extends to any closed projective subvariety $V$ of a projective space $\mathbb{P}^n$. Indeed, we cover $\mathbb{P}^n$ by finitely many open affine spaces $\{U_i\}_i$, and then apply Theorem~\ref{main result} to each $V\cap U_i$ (which is a closed subvariety of the affine space $U_i$).
\end{remark}

We prove Theorem~\ref{main result} using the same strategy employed in \cite{Ghioca-bogo3}, only that this time we replace the Frobenius endomorphism by a suitable automorphism $\sigma$ of $\overline{k(t)}$. We show that for any point $P\in V(\overline{k(t)})$, either $\h(P)$ is uniformly bounded from below away from $0$, or $P^{\sigma}\in V(\overline{k(t)})$. If the latter occurs generically, then $V$ is invariant under $\sigma$ and therefore it is defined over the fixed field of $\sigma$, which is $k$.

Theorem~\ref{main result} yields a similar result for an arbitrary (finite) transcendence degree function field. Indeed, let $k$ be an algebraically closed field, and let $K/k$ be a finite transcendence degree function field. Let $t_1,\dots, t_r\in K$ be algebraically independent elements such that $K/k(t_1,\dots, t_r)$ is a finite extension. For each $i=1,\dots, r$ we let $K_i$ be the algebraic closure of $k(t_1,\dots, t_i)$ in $K$. Then $K_r=K$; we also let $K_0:=k$. 
Let $V\subset \bA^n$ be an affine variety defined over $K$, and assume it's not defined over $k$ (otherwise Theorem~\ref{main result} and its consequences hold trivially). Then there exists a smallest (positive) integer $i$ such that $V$ is defined over $K_i$ (but it's not defined over $K_{i-1}$). We let $\h$ be the Weil height constructed with respect to the function field $K_i/K_{i-1}$ (which is a function field of transcendence degree equal to $1$). Then Theorem~\ref{main result} yields that there exists a positive real number $\epsilon:=\epsilon(V)$ such that if $P\in V(\overline{K_i})$ and $h(P)<\epsilon$, then $P\in W(\overline{K_i})$, where $W$ is the largest subvariety of $V$ defined over $\overline{K_{i-1}}$.

\medskip

\emph{Acknowledgments.} We thank the Institute of Mathematics Academia Sinica for its hospitality, and the organizers Liang-Chung Hsia and Tzu-Yueh Julie Wang of the conference on ``Diophantine Problems and Arithmetic Dynamics'' from Taipei for creating a stimulating research environment where the results of this paper were disseminated.

\section{Proof of our main result}
\label{se:proof}

Our proof follows the strategy from \cite{Ghioca-bogo3}; when the proof is identical with the one from \cite{Ghioca-bogo3} we refer to our earlier paper, otherwise we present the argument entirely.  
Unless otherwise stated, all our subvarieties are closed; we continue with our notation from Theorem~\ref{main result}. We start with a definition. 
\begin{definition}
\label{D:reduced polynomials}
We call \emph{reduced} a non-constant polynomial 
$$f\in k[t][X_1,\dots,X_n],$$ 
whose coefficients $a_i$ have no non-constant common divisor in $k[t]$. We define the \emph{height}  $\h(f)$ of the polynomial $f$ as the maximum of the degrees of the coefficients $a_i\in k[t]$ of $f$.
\end{definition}

For some integer $M>1$, let  $\sigma:=\sigma_M$ be an automorphism of $\overline{k(t)}$ which fixes the elements of $k$, and maps $t$ into $t^M$.

\begin{lemma}
\label{height of sigma}
For each $x\in \overline{k(t)}$, we have $h(\sigma(x))=M\cdot h(x)$.
\end{lemma}

\begin{proof}
Let 
$$L:=\bigcup_{n\ge 1} k\left(t^{\frac{1}{M^n}}\right).$$ 
We claim that $\sigma$ restricts to an automorphism of $L$. First of all, it is clear that for each positive integer $n$, there exists an $M$-th root of unity $\zeta_n\in k$ such that $\sigma(t^{1/M^n})=\zeta_n t^{1/M^{n-1}}$. So, $\sigma(\zeta_n^{-1}t^{1/M^n})=t^{1/M^{n-1}}$ showing that indeed $\sigma$ restricts to an automorphisms of $L$ (note that $k$ is algebraically closed).

Clearly, we may assume $x\ne 0$. 
Let $f\in L[z]$ be a polynomial of minimal degree such that $f(x)=0$. Since $f$ has finitely many coefficients (say, $\deg(f)=d\ge 1$), then there exists $N\in\bN$ such that $f$ has all its coefficients in $k\left(t^{1/M^N}\right)$. For the sake of simplifying the notation, we let $T:=t^{1/M^N}$. Since $k[T]$ is a PID, we may assume $f\in k[T][z]$ and moreover, the coefficients of $f$ are all relatively prime. Furthermore, $f$ is irreducible in $k[T][z]$; also let $D:=\h_T(f)$ be the maximum of the degrees (in $T$) of the coefficients of $f$. So, applying \cite[Lemma 2.1]{Derksen-Masser}, we conclude that
$$\h(x)=\frac{D}{dM^N}.$$
An  observation regarding our formula above and \cite[Lemma 2.1]{Derksen-Masser}: because our height is defined relative to $k(t)$, while in \cite[Lemma 2.1]{Derksen-Masser}  the height is computed relative to the field $k(T)=k\left(t^{1/M^N}\right)$, the factor $M^N$ appears in the denominator of our formula.

On the other hand, we claim that $f^\sigma\in L[z]$ is also irreducible, where $f^\sigma$ is the polynomial obtained by applying $\sigma$ to each coefficient of $f$. Indeed, if $f^\sigma$ were reducible over $L$, then there exist nonconstant polynomials $g,h\in L[z]$ such that $f^\sigma = g\cdot h$. But then $f=g^{\sigma^{-1}}\cdot h^{\sigma^{-1}}$, and $g^{\sigma^{-1}},h^{\sigma^{-1}}\in L[z]$ which thus contradicts the hypothesis that $f$ is irreducible in $L[z]$. Moreover, the coefficients of $f^{\sigma}$ are relatively prime. Indeed, because the coefficients $\{a_i\}_{0\le i\le d}$ of $f$ are relatively prime there exist $b_i\in k[T]$ such that $\sum_{i=0}^d a_i b_i = 1$, and so, $\sum_{i=0}^d \sigma(a_i)\sigma(b_i)=1$ showing that also the coefficients $\sigma(a_i)$ of $f^\sigma$ are relatively prime. Hence, applying again \cite[Lemma 2.1]{Derksen-Masser} we conclude that
$$\h(\sigma(x))=\frac{M\cdot D}{d M^N},$$
since $\deg_T(\sigma(a_i))=M\cdot \deg_T(a_i)$ for each $i$. Thus, indeed $\h(\sigma(x))=M\cdot \h(x)$.
\end{proof}

\begin{corollary}
\label{cor: height of affine points}
For each $x_1,\dots,x_n\in \overline{k(t)}$ we have
$$\h((\sigma(x_1),\dots, \sigma(x_n)))\le nM\h((x_1,\dots,x_n)).$$
\end{corollary}

\begin{proof}
Let $K$ be a finite extension of $k(t)$ containing each $x_i$ and each $\sigma(x_i)$. Using Lemma~\ref{height of sigma}, we obtain
\begin{eqnarray*}
\h((\sigma(x_1),\dots,\sigma(x_n)))\\
& = \frac{1}{[K:k(t)]}\sum_{w\in\Omega_K} \max\{0,-w(\sigma(x_1)),\dots, -w(\sigma(x_n))\}\\
& \le \frac{1}{[K:k(t)]}\sum_{w\in \Omega_K} \sum_{i=1}^n \max\{0, -w(\sigma(x_i))\\
& = \frac{1}{[K:k(t)]}\sum_{i=1}^n \h(\sigma(x_i))\\
& = \frac{M}{[K:k(t)]}\sum_{i=1}^n \h(x_i)\\
& = \frac{M}{[K:k(t)]}\sum_{w\in\Omega_K} \sum_{i=1}^n \max\{0, - w(x_i)\}\\
& \le \frac{Mn}{[K:k(t)]}\sum_{w\in\Omega_K} \max\{0,-w(x_1),\dots, -w(x_n)\}\\
& = Mn\h((x_1,\dots, x_n)),
\end{eqnarray*}
as desired.
\end{proof}

The following result is also an easy corollary of Lemma~\ref{height of sigma}.
\begin{lemma}
\label{fixed field of sigma}
The fixed field of $\sigma$ is $k$.
\end{lemma}

\begin{proof}
Let $x\in\overline{k(t)}$ such that $\sigma(x)=x$. Then by Lemma~\ref{height of sigma} we have that $\h(x)=\h(\sigma(x))=M\h(x)$; so $\h(x)=0$ (because $M>1$). Therefore $x\in k$ since they are the only points in $\overline{k(t)}$ of height equal to $0$.
\end{proof}

The following result is key for our proof, and it is similar to \cite[Lemma~3.2]{Ghioca-bogo3}.

\begin{lemma}
\label{L:L1}
Let $f\in k[t][X_1,\dots,X_n]$ be a reduced polynomial of total degree $d$. Let $M$ be an integer satisfying $M\ge \max\{1,2\h(f)\}$, and let $\sigma$ be an automorphism of $\overline{k(t)}$ such that $\sigma$ restricts to the identity morphism on $k$, and $\sigma(t)=t^M$. If $(x_1,\dots,x_n)\in\mathbb{A}^n_{\overline{k(t)}}$ satisfies $f(x_1,\dots,x_n)=0$, then either
$$\h(x_1,\dots,x_n)\ge\frac{1}{2dn}$$
or
$$f(\sigma(x_1),\dots, \sigma(x_n))=0.$$
\end{lemma}

\begin{proof}
Let $(x_1,\dots,x_n)\in\mathbb{A}_{\overline{k(t)}}^n$ be a zero of $f$. We let $f=\sum_i a_iM_i$,
where the $a_i$'s are the nonzero coefficients of $f$ and the $M_i$'s are the corresponding monomials of $f$. For each $i$, we let $m_i:=M_i(x_1,\dots,x_n)$.

Assume $f(\sigma(x_1),\dots,\sigma(x_n))\ne 0$. 

We let $K=k(t,x_1,\dots,x_n)$. If $\zeta= f(\sigma(x_1),\dots,\sigma(x_n))$, then (because $\zeta\ne 0$)
\begin{equation}
\label{E:sum formula for zeta}
\sum_{w\in M_K}w(\zeta)=0.
\end{equation}

Because $f(x_1,\dots,x_n)=0$, we get $\zeta=\zeta-\sigma(f(x_1,\dots,x_n))$ and so,
\begin{equation}
\label{E:trick}
\zeta=\sum_i (a_i-\sigma(a_i))\cdot \sigma(m_i).
\end{equation}

\begin{Claim}
\label{C:divizibilitate}
For every $g\in k[t]$, $\left(t^M-t\right)\mid\left(\sigma(g)-g\right)$.
\end{Claim}

\begin{proof}[Proof of Claim \ref{C:divizibilitate}.]
Let $g:=\sum_{j=0}^m b_jt^j$, with $b_j\in k$. Then $\sigma(g)=\sum_{j=0}^m b_jt^{jM}$. The proof of Claim~\ref{C:divizibilitate} is immediate because for every $j\in\mathbb{N}$, we have $\sigma(b_j)=b_j$ and  $\left(t^{M}-t\right)\mid\left(t^{jM}-t^j\right)$.
\end{proof}

Using the result of Claim ~\ref{C:divizibilitate} and equation \eqref{E:trick}, we get
\begin{equation}
\label{E:alternative zeta}
\zeta=(t^{M}-t)\cdot \sum_i b_i\sigma(m_i),
\end{equation}
where $b_i=\frac{a_i-a_i^{M}}{t^{M}-t}\in k[t]$. Let $S$ be the set of valuations $w\in M_K$ such that $w$ lies above each place of $K(t)$ corresponding to a root of $t^{M}-t$. For each $w\in S$,

\begin{equation}
\label{E:first inequality}
w(\zeta)\ge  w(t^{M}-t)-d\max\{0,-w(\sigma(x_1)),\dots, -w(\sigma(x_n))\},
\end{equation}
because for each $i$, $w(b_i)\ge 0$ (as $b_i\in k[t]$ and $w$ does not lie over $v_{\infty}$) and the total degree of $M_i$ is at most $d$.

For each $w\in M_K\setminus S$, because $\zeta=\sum_i a_i\sigma(m_i)$, we have
\begin{equation}
\label{E:second inequality}
w(\zeta)\ge -\max\{0,\max_i-w(a_i)\} -d\max\{0,-w(\sigma(x_1)),\dots,-w(\sigma(x_n))\}.
\end{equation}

Adding all inequalities from \eqref{E:first inequality} and \eqref{E:second inequality} we obtain
\begin{equation}
\label{E:final inequality}
0=\sum_{w\in M_K}\frac{w(\zeta)}{[K:k(t)]}\ge -\h(f)-d\h((\sigma(x_1),\dots,\sigma(x_n)))+\sum_{\substack{w\in M_K\\w(t^{M}-t)>0}}\frac{ w(t^{M}-t)}{[K:k(t)]}.
\end{equation}
By the coherence of the valuations on $\Kbar$ (see \eqref{coherence of the place v}), we have 
$$\sum_{\substack{w\in M_K\\w(t^{M}-t)>0}}\frac{ w(t^{M}-t)}{[K:k(t)]}=\sum_{\substack{v\in M_{k(t)}\\v(t^{M}-t)>0}} v(t^{M}-t)=-v_{\infty}(t^{M}-t)=M.$$
Thus, inequality \eqref{E:final inequality} yields
$$0\ge-\h(f)-d\h((\sigma(x_1),\dots,\sigma(x_n)))+M$$
and so, using Corollary~\ref{cor: height of affine points} we obtain
\begin{equation}
\label{E:rough inequality}
dnM\h((x_1,\dots,x_n))\ge M-h(f).
\end{equation}
Because $M$ was chosen such that $M\ge 2\h(f)$, we conclude that
\begin{equation}
\label{E:the inequality}
\h((x_1,\dots,x_n))\ge\frac{1}{2dn}.
\end{equation}
as desired.
\end{proof}

\begin{lemma}
\label{L:L2}
Let  $f\in \overline{k(t)}[X_1,\dots,X_n]$ be a nonzero polynomial, and let $\sigma$ be an automorphism of $\overline{k(t)}$ fixing pointwise $k$ and mapping $t$ into $t^M$ (for some integer $M>1$). If $f(X_1,\dots,X_n)\mid f^{\sigma^{-1}}(X_1,\dots,X_n)$,  
then there exists a nonzero $a\in\overline{k(t)}$ such that  $a\cdot f\in k[X_1,\dots,X_n]$.
\end{lemma}

\begin{proof}
Let $Z:=Z(f)$ be the zero set for $f$. The hypothesis on $f$ shows that for every $P\in Z(\overline{k(t)})$, we have $P^\sigma\in Z(\overline{k(t)})$. Hence $Z$ is invariant under $\sigma$, and therefore $Z$ is defined over the fixed field  of (a power of) $\sigma$, which is $k$ (by Lemma~\ref{fixed field of sigma}).  

Essentially, the reasoning is as follows:  $Z$ and $Z^{\sigma}$ (the hypersurface given by the equation $f^{\sigma^{-1}}=0$) have the same number of irreducible components, and thus each irreducible component of $Z$ is fixed by (a power of) $\sigma$. Since the fixed subfield for each power of $\sigma$ is $k$ (by Lemma~\ref{fixed field of sigma}), without loss of generality we may assume each irreducible component of $Z$ is fixed by $\sigma$. In other words, for each irreducible polynomial $g$ dividing $f$ we have $Z(g)=Z\left(g^{\sigma^{-1}}\right)$; also, we may assume at least one of the coefficients of $g$ equals $1$ (we simply divide $g$ by one of its nonzero coefficients from $\overline{k(t)}$). So, there exists $b\in\overline{k(t)}$ such that $g^{\sigma}=b\cdot g$; but because one of the coefficients of $g$ equals $1$, we conclude that $b=1$. Hence $g$ has all its coefficients in $k[t]$. Since
$$f=A\cdot \prod_i g_i^{e_i},$$
for some $A\in\overline{k(t)}$, we obtain the desired conclusion with $a=A^{-1}$.
\end{proof}

\begin{lemma}
\label{L:L3}
Let $V\subset\mathbb{A}^n$ be a proper affine $\overline{k(t)}$-subvariety.  Then there exists a positive constant $C$, depending only on $V$, and there exists a proper affine $k$-subvariety $Z\subset\mathbb{A}^n$, which also depends only on $V$, such that for every $P\in V(\overline{k(t)})$, either $P\in Z(\overline{k(t)})$ or $\h(P)\ge C$.
\end{lemma}

\begin{remark}
The only difference between Lemma~\ref{L:L3} and Theorem~\ref{main result} is that we do not require $Z$ be contained in $V$.
\end{remark}

\begin{proof}[Proof of Lemma~\ref{L:L3}.]
The proof follows the arguments from  \cite[Lemma~3.5]{Ghioca-bogo3}. We proceed by induction on $n$.  The case $n=1$ is obvious, because any subvariety of $\mathbb{A}^1$, different from $\mathbb{A}^1$, is a finite union of points. Thus we may take $Z=V(k)$, (which is also a finite union of points) and $C:=\min_{P\in (V\setminus Z)(\overline{k(t)})}\h(P)$ (if there are no points in $V(\overline{k(t)})\setminus V(k)$, then we may take $C=1$, say). We note that in this case ($n=1$) we actually proved  Theorem \ref{main result}, because the variety $Z$ that we chose is a subvariety of $V$. 

We assume Lemma \ref{L:L3} holds for $n-1$ and we prove it for $n$ ($n\ge  2$). Let $K$ be a finite field extension of $k(t)$ (of minimal degree) such that $V$ is defined over $K$. Let $p^m$ be the inseparable degree of the extension $K/k(t)$ ($m\ge 0$). Let 
$$V_1=\bigcup_{\sigma}V^{\sigma},$$
where $\sigma$ denotes any field morphism $K\rightarrow \overline{k(t)}$ which fixes $k(t)$. The variety $V_1$ is a $k(t^{1/p^m})$-variety (note that $k$ is algebraically closed). Also, $V_1$ depends only on $V$. Thus, if we prove Lemma~\ref{L:L3} for $V_1$, then our result will hold also for $V\subset V_1$. Hence we may and do assume that $V$ is defined over $k(t^{1/p^m})$.

Assume $m>0$; then $k$ has positive characteristic. We let $F$ be the Frobenius corresponding to $\mathbb{F}_p$. The variety $V'=F^mV$ is a $k(t)$-variety, which depends only on $V$. Assume we proved Lemma~\ref{L:L3} for $V'$ and let $C'$ and $Z'$ be the positive constant and the $k$-variety, respectively, associated to $V'$, as in the conclusion of Lemma~\ref{L:L3}. Let $P\in V(\overline{k(t)})$. Then $P':=F^m(P)\in V'(\overline{k(t)})$. Thus, either 
$$\h(P')\ge C'\text{ or}$$ 
$$P'\in Z'(\overline{k(t)}).$$
In the former case, because $\h(P)=\frac{1}{p^m}\h(P')$, we obtain a lower bound for the height of $P$, depending only on $V$ (note that $m$ depends only on $V$). In the latter case, if we let $Z$ be the $k$-subvariety of $\mathbb{A}^n$, obtained by extracting the $p^m$-th roots of the coefficients of a set of polynomials (defined over $k$) which generate the vanishing ideal for $Z'$, we get $P\in Z(\overline{k(t)})$. By its construction, $Z$ depends only on $V$ and so, we obtain the conclusion of Lemma~\ref{L:L3}.

Thus, from now on in this proof, we assume $V$ is a $k(t)$-variety. We fix a set of defining polynomials for $V$ which contains polynomials $P_i\in k[t][X_1,\dots,X_n]$ for which $$\max_i\deg(P_i)$$ is minimum among all possible sets of defining polynomials for $V$ (where $\deg P_i$ is the total degree of $P_i$). We may assume all of the polynomials we chose are reduced. If all of them have coefficients in $k$, then Lemma ~\ref{L:L3} holds with $Z=V$ and $C$ any positive constant. Once again, in this case, Theorem~\ref{main result} holds.

Assume there exists a reduced polynomial $f\notin k[X_1,\dots,X_n]$ in the fixed set of defining equations for $V$. Let $L:=\bigcup_{\ell\ge 1} k\left(t^{1/\ell!}\right)$, and let $\{f_i\}_i$ be the set of all the  $L$-irreducible factors of $f$. For each $i$ let $H_i$ be the zero set of $f_i$. Then $V$ is contained in the finite union $\cup_i H_i$.
The polynomials $f_i$ depend only on $f$. Thus it suffices to prove Lemma~\ref{L:L3} for each $H_i$. Hence we may and do assume $V$ is the zero set of an $L$-irreducible polynomial $f\notin k[X_1,\dots,X_n]$. For the sake of simplifying our notation, we let $T:=t^{1/\ell!}$ such that $f\in k[T][X_1,\dots, X_n]$; moreover we may assume $f$ is reduced (over $k[T]$). Again, note that $\ell$ depends only on $V$. 
Let $M$ and $\sigma$ be as in Lemma~\ref{L:L1} with respect to the polynomial $f$ defined over $k[T]$ (note that we still ask that $\sigma(t)=t^M$).

Let $P=(x_1,\dots,x_n)\in V(\overline{k(t)})$.  We apply Lemma~\ref{L:L1} to $f$ and $P$ and conclude that either 
\begin{equation}
\label{E:height inequality}
\h(P)\ge\frac{1}{2n\ell!\deg(f)}
\end{equation}
or 
\begin{equation}
\label{E:alternative}
f(\sigma(x_1),\dots,\sigma(x_n))=0.
\end{equation}
If \eqref{E:height inequality} holds, then we obtained a good lower bound for the height of $P$ (depending only on $V$). 

Assume \eqref{E:alternative} holds. Because $f$ is an irreducible and reduced polynomial, whose coefficients are not all in $k$, Lemma~\ref{L:L2} yields that $f(X_1,\dots,X_n)$ cannot divide $f^{\sigma^{-1}}(X_1,\dots,X_n)$. Indeed, if $f\mid f^{\sigma^{-1}}$ then there exists $a\in\overline{k(t)}$ such that $af\in k[X_1,\dots,X_n]$. But this yields that for each nonzero coefficient $c_i$ of $f$ we have $\sigma(c_i)=ac_i$; so $a\in k(T)$ because each $c_i,\sigma(c_i)\in k[T]$. On the other hand, the $c_i$'s are relatively prime and therefore the $\sigma(c_i)$'s are relatively prime, which yields that $a\in k$. But then because  $\deg_T(\sigma(c_i))=M\cdot \deg_T(c_i)$, we conclude that each $c_i$ is in $k$, which is a contradiction with the fact that $f\notin k[X_1,\dots, X_n]$.

We know $f$ has more than one monomial because it is reduced and not all of its coefficients are in $k$. 
Without loss of generality, we may assume $f$ has positive degree in $X_n$. Because $f$ is $L$-irreducible and $f$ does not divide $f^{\sigma^{-1}}$ (which is also defined over $L$), we conclude that $f$ and $f^{\sigma^{-1}}$ are relatively prime. So, the resultant $R$ of the polynomials $f(X_1,\dots,X_n)$ and $f^{\sigma^{-1}}(X_1,\dots,X_n)$ with respect to the variable $X_n$ is nonzero. Moreover, $R$ depends only on $f$.

The nonzero polynomial $R\in k(t^{1/M\ell !})[X_1,\dots,X_{n-1}]$ (since $f^{\sigma^{-1}}$ is defined over $k(t^{1/M\ell!})$) vanishes on $(x_1,\dots,x_{n-1})$. Applying the induction hypothesis to the hypersurface $R=0$ in $\mathbb{A}^{n-1}$, we conclude there exists a proper $k$-subvariety $Z_1\subset \bA^{n-1}$,  depending only on $R$ (and so, only on $V$) and there exists a positive constant $C$, depending only on $R$ (and so, only on $V$) such that either
\begin{equation}
\label{E:1stdichotomy1}
\h(x_1,\dots,x_{n-1})\ge C\text{ or}
\end{equation}
\begin{equation}
\label{E:1stdichotomy2}
(x_1,\dots,x_{n-1})\in Z_1(\overline{k(t)}).
\end{equation}
If \eqref{E:1stdichotomy1} holds, then $\h(x_1,\dots,x_{n-1},x_n)\ge\h(x_1,\dots,x_{n-1})\ge C$ and we have a height inequality as in the conclusion of Lemma ~\ref{L:L3}. If \eqref{E:1stdichotomy2} holds, then $(x_1,\dots,x_n)\in \left(Z_1\times\mathbb{A}^1\right)(\overline{k(t)})$ and $Z_1\times\mathbb{A}^1$ is a $k$-variety, strictly contained in $\mathbb{A}^n$, as desired in Lemma~\ref{L:L3}. This proves the inductive step and concludes the proof of Lemma ~\ref{L:L3}.
\end{proof}

The following result is proved in \cite[Corollary~2.4]{Ghioca-bogo3}; essentially the proof relies on the fact that if $(C_i,Z_i)$ are two pairs as in the conclusion of Lemma~\ref{L:L3} (for $i=1,2$), then $(\min\{C_1,C_2\},Z_1\cap Z_2)$ is another pair satisfying the conclusion of Lemma~\ref{L:L3}.
\begin{corollary} 
\label{C:L3}
Let $X$ be a proper subvariety of $\mathbb{A}^n$ defined over $\overline{k(t)}$. There exists a positive constant $C$ and a proper subvariety $Z\subset\mathbb{A}^n$ defined over $k$, such that the pair $(C,Z)$ satisfies the conclusion of Lemma~\ref{L:L3}, and moreover $Z$ is minimal with this property (with respect to the inclusion of subvarieties of $\mathbb{A}^n$). 
\end{corollary}

Then Theorem \ref{main result} follows from Corollary~\ref{C:L3} exactly as the proof of \cite[Theorem~2.2]{Ghioca-bogo3}; the only difference is that $\mathbb{F}_p$ is replaced by $k$.

\end{document}